\documentclass{amsart}
\usepackage{amsmath, amscd, amssymb, amsthm}
\usepackage{bbm}
\usepackage{latexsym}
\usepackage{amsfonts}
\usepackage{graphicx}
\usepackage{xcolor}
\usepackage{needspace}
\usepackage[all,cmtip]{xy}
\usepackage[colorlinks,linkcolor=blue,breaklinks=true,urlcolor=blue,citecolor=blue,anchorcolor=blue,pagebackref]{hyperref}%
\usepackage{geometry}
\newtheorem{theorem}{Theorem}
\newtheorem{lemma}{Lemma}
\newtheorem{corollary}[theorem]{Corollary}
\newtheorem{remark}{Remark}
\newtheorem{proposition}{Proposition}

\newtheorem*{example}{Example}

\renewcommand*\backref[1]{}
\renewcommand*\backrefalt[4]{ \ifcase #1 \or (cited on page #2) \else (cited on pages #2) \fi}

\newcommand{\be}{\begin{equation}}
\newcommand{\ee}{\end{equation}}
\newcommand{\bea}{\begin{eqnarray}}
\newcommand{\eea}{\end{eqnarray}}

\def\XXint#1#2#3{{\setbox0=\hbox{$#1{#2#3}{\int}$ }
\vcenter{\hbox{$#2#3$ }}\kern-.6\wd0}}

\begin{document}

\title[Orthogonal complex structures on flat tori]{Orthogonal complex structures on flat tori}

\author{Qingsong Wang}
\address{Qingsong Wang. Hal{\i}c{\i}o\u{g}lu Data Science Institute, University of California San Diego, La Jolla, CA 92093, USA}
\email{qswang92@gmail.com}
\thanks{Zheng is the corresponding author. He is partially supported by National Natural Science Foundation of China
with the grant No.~12471039 and  12141101, and by the 111 Project D21024.}

\author{Shing-Tung Yau}
\address{Shing-Tung Yau. Yau Mathematical Sciences Center, Tsinghua University, Beijing 100084, China}
\email{{styau@mail.tsinghua.edu.cn}}

\author{Fangyang Zheng}
\address{Fangyang Zheng. School of Mathematical Sciences, Chongqing Normal University, Chongqing 401331, China}
\email{20190045@cqnu.edu.cn; franciszheng@yahoo.com}


\begin{abstract}
  We classify compact Hermitian manifolds with flat Levi-Civita connection. This is equivalent to the classification of all orthogonal complex structures on flat tori. It generalized the work of Khan, Yang, and Zheng in 2017 where they solved the case in complex dimension three.
 \end{abstract}

\maketitle

\tableofcontents

\markleft{Qingsong Wang, Shing-Tung Yau, and Fangyang Zheng}
\markright{Orthogonal complex structures on flat tori}

\section{Introduction and statement of results}\label{intro}

Given a Hermitian manifold $(M^n,g)$, there are three canonical metric connections that are extensively studied: the Levi-Civita (Riemannian) connection $\nabla$, the Chern connection $\nabla^c$, and the Strominger-Bismut connection $\nabla^b$ (\cite{Strominger, Bismut}). When $g$ is K\"ahler, these three connections coincide, but when $g$ is not K\"ahler, the connections are mutually distinct, and we have three different kinds of geometry.

The first thing to understand in these three geometries is the structure of
the corresponding flat spaces: which compact Hermitian manifolds have an
everywhere-flat Levi-Civita, Chern, or Strominger--Bismut connection?
For the Chern connection, Boothby's classic result \cite{Boothby} in 1958 says that compact Chern flat manifolds are exactly compact quotients of complex Lie groups. In 2020, Wang, Yang, and Zheng proved \cite{WYZ} that compact Strominger-Bismut flat manifolds are exactly compact quotients of Samelson spaces \cite{Samelson}, which means a connected, simply connected Lie group equipped with a bi-invariant metric and a compatible left-invariant complex structure. Such Lie groups are products of the vector group ${\mathbb R}^k$ with compact semisimple Lie groups by Milnor's lemma \cite{Milnor}, and all Samelson spaces are known to be Strominger-Bismut flat by the work of Pittie \cite{Pittie}.

The same question can be asked for the Levi-Civita connection. Let $(M^n,g)$ be a compact Hermitian manifold with flat Levi-Civita connection. Since as a Riemannian manifold it is compact and flat, we know by Bieberbach theorem that there exists a finite unbranched cover of $M$ which is a flat torus $T^{2n}_{\mathbb R}$. Thus, up to a finite cover, the question is to find all complex structures on a flat torus that are compatible with the flat metric, or, in Simon Salamon's terminology \cite{Salamon}, all the {\em orthogonal complex structures} on flat tori. This is the main task of the present article.

{We call an orthogonal complex structure $J$ on a flat torus {\em constant} if $\nabla J=0$; equivalently, the given flat Hermitian metric $g$ is K\"ahler. The torus is then a flat complex torus.}

For $n\leq2$, every orthogonal complex structure on a flat torus is
constant. For $n\geq3$, nonconstant examples are provided by the BSV tori
\cite{KYZ}, named after Borisov, Salamon, and Viaclovsky \cite{BSV}.
Their affine generalizations will be defined in Section~\ref{sec:affine-bsv}.
An affine BSV torus $(M^n,J,g)$ is a flat real $2n$-torus equipped with a
holomorphic Riemannian submersion $\pi:M\to B$ onto an abelian variety
of complex dimension $0<r<n$. Its fibers are flat complex tori, isometric
as Riemannian manifolds to a fixed real $(2n-2r)$-torus, but their complex
structures vary nontrivially with the base point: the holomorphic family
is not isotrivial.

The {\em twistor rank} of an orthogonal complex structure is the generic
complex rank of its twistor map, which records $J$ in a parallel frame.
For an affine BSV torus, the twistor rank equals the base dimension $r$.
(The degenerate case $r=0$ is exactly the constant, K\"ahler case and is
excluded from the affine BSV definition.)

\begin{theorem} \label{thm1}
Let $(M^n,J,g)$ be a compact Hermitian manifold of complex dimension
$n\geq3$ with flat Levi-Civita connection. Then $M$ has a finite unbranched
cover which, with the lifted Hermitian structure, is either a flat complex
torus or an affine BSV torus. In the latter case, its twistor rank $r$
satisfies
\[
 0<r<n,\qquad 2r\leq(n-r)(n-r-1).
\]
\end{theorem}

{For nonconstant structures, the rank bound gives $r=1$ when $n=3$
or $4$, and $r\leq n-3$ when $n\geq5$. This numerical condition is
necessary but not sufficient for existence: the pair $(n,r)=(10,6)$
satisfies it with equality, but no orthogonal complex structure on a flat
torus of complex dimension $10$ has twistor rank $6$; see the final example
in Section~\ref{sec:affine-bsv}. An exact vector-bundle characterization
of the possible positive ranks is given in
Theorem~\ref{thm:rank-realization}.}

{The affine BSV tori have the following geometric
properties.} Recall that a compact complex manifold is in Fujiki
${\mathcal C}$ class \cite{Fujiki}, if it is bimeromorphic to a compact
K\"ahler manifold.

\begin{proposition}\label{prop:bsv-geometry}
Let $(M^n,J,g)$ be an affine BSV torus of twistor rank $r$.
The metric $g$ is balanced, and $M$ has Kodaira dimension $-\infty$.
Moreover, $M$ is not in Fujiki class ${\mathcal C}$, so it is neither
K\"ahlerian nor Moishezon, and its algebraic dimension satisfies
\[
 r\leq a(M)<n.
\]
\end{proposition}

A Hermitian metric is called {\em pluriclosed} if its fundamental form
$\eta$ satisfies $\partial\bar\partial\eta=0$.
The following theorem strengthens the non-K\"ahler conclusion of
Proposition~\ref{prop:bsv-geometry}.

\begin{theorem}[No pluriclosed metrics on affine BSV tori]
\label{thm:no-pluriclosed}
Let $M^n$ be an affine BSV torus of twistor rank $r$.
Then $M$ admits no pluriclosed Hermitian metric.
\end{theorem}

Together with balancedness, this verifies the Fino--Vezzoni conjecture
\cite{FinoVezzoni,FinoVezzoni1} for affine BSV tori.

{The horizontal--vertical splitting also imposes an algebraic restriction on the Chern torsion.} Given a Hermitian manifold $(M^n,g)$, denote by $T$ the Chern torsion tensor. It is a linear map from $\Lambda^2V$ into $V$, where $V=T_p^{1,0}M$ is the holomorphic tangent space of $M$ at any fixed point $p\in M$. Denote by $im(T)\subset V$ the image space of this map, and by $ann(T)\subset V$ the `kernel' space of $T$, namely, the set of all $v\in V$ such that $T(v,w )=0$ for any $w\in V$.

\begin{corollary} \label{cor2}
Let $(M^n,g)$ be an affine BSV torus. Denote by ${\mathcal H}^{1,0}$ the
$(1,0)$-part of the horizontal distribution of $\pi$. Then at every point in
$M$, $im(T) \subset {\mathcal H}^{1,0} \subset ann(T)$. In particular, the
Chern torsion $T$ satisfies
\begin{equation} \label{eq:Tsquare}
\sum_{s=1}^n T^{\ell}_{is} T^s_{jk}=0, \ \ \ \ \ \forall \ 1\leq i,j,k,\ell \leq n,
\end{equation}
under any unitary frame $e$, where $T(e_i,e_k)=\sum_j T^j_{ik}e_j$.
\end{corollary}

{In \cite{KYZ}, the proof in dimension $3$ starts by establishing
\eqref{eq:Tsquare} through local computation. Here the global geometry of
the twistor map gives Theorem~\ref{thm1}, and the torsion identity follows
from the affine splitting.}

{Theorem~\ref{thm1} also fits into the broader study of flat canonical metric connections.} For any $t\in {\mathbb R}$, the connection $D^t=\frac{1+t}{2}\nabla^c  + \frac{1-t}{2}\nabla^b$ is called the {\em $t$-th Gauduchon connection.} {Lafuente and Stanfield proved in \cite{LS} that, for $t\neq\pm1$, a compact Hermitian manifold with flat $D^t$ must be K\"ahler.} Partial results were obtained by Fu and Zhou \cite{FuZhou}  and by Yang and Zheng \cite{YangZ1} earlier for $t$ values outside a critical interval. More generally, in \cite{ZhaoZ}, Zhao and Zheng considered the two parameter family of {\em canonical metric connections}
$$ D^t_s=(1-s)D^t+ s\nabla, \ \ \ \ \ (t,s) \in \Omega = ({\mathbb R}^2 \setminus \{ s=1\})\cup \{ (0,1)\}. $$
The following points in $\Omega$ are special:
$$ P=\big\{ (1,0), \ (-1,0), \ (0,1), \ (0,-1), \ (-1,2), \ (\frac13 , -2)\big\}. $$
The corresponding connections are respectively the Chern $\nabla^c$, the Strominger-Bismut $\nabla^b$, the Levi-Civita $\nabla$, the anti-Levi-Civita $\nabla'$, and two special connections $\nabla^+$ and $\nabla^-$. $\nabla'$ is closely related to and moves in sync with $\nabla$, while $\nabla^+$ and $\nabla^-$ are closely related to and move in sync with $\nabla^b$ (we refer the readers to \cite{ZhaoZ} for more details on this). Combining the aforementioned result of Lafuente-Stanfield with the result of Zhao-Zheng \cite[Theorem 4]{ZhaoZ}, one can state the following:

\begin{theorem}[\cite{LS, ZhaoZ}]
For any $(t,s)\in \Omega \setminus P$, if $(M^n,g)$ is a compact Hermitian manifold such that its canonical metric connection $D^t_s$ is flat, then $g$ must be K\"ahler.
\end{theorem}

{Theorem~\ref{thm1} describes the compact Levi-Civita-flat case within this family.}

Section~\ref{sec:affine-bsv} gives the affine BSV construction, its
lattice-splitting criterion, and the vector-bundle characterization in
Theorem~\ref{thm:rank-realization}.
Section~\ref{sec:classification} proves Theorem~\ref{thm1},
Proposition~\ref{prop:bsv-geometry}, and Corollary~\ref{cor2}.
Section~\ref{sec:pluriclosed} proves Theorem~\ref{thm:no-pluriclosed}.

\vspace{0.3cm}

\section{Affine BSV tori}\label{sec:affine-bsv}

Let $V\cong {\mathbb R}^{2n}$ be a vector space equipped with an inner product $g=\langle \cdot , \cdot \rangle$ and an orientation. Denote by $Z_n$ the set of all almost complex structures on $V$ compatible with $g$ and the orientation. It is one component of the maximal isotropic Grassmannian and is a compact Hermitian symmetric space:
$$ Z_n = \mbox{OGr}^+(n, V_{\mathbb C}) \cong SO(2n)/U(n). $$
Here $V_{\mathbb C}=V\otimes {\mathbb C}\cong {\mathbb C}^{2n}$. Suppose $\Lambda \cong {\mathbb Z}^{2n}$ is a lattice in $V$. Then $M=V/\Lambda $ becomes a flat oriented $2n$-torus. The metric {descends from} $g$, and we continue to denote it by $g$. Let $J$ be an almost complex structure on $M$ compatible with both the metric and the orientation. Fix a parallel orthonormal frame on $M$ and $V$. Then under the frame each tangent space $T_pM$ is identified with $V$, so $J_p$ becomes an element in $Z_n$. The induced map
$$  \tau : M \rightarrow Z_n, \ \ \ p \mapsto J_p  $$
is called the {\em twistor map.} {The flat twistor integrability
criterion \cite[Proposition~4.5]{BSV} says that $J$ is integrable if and
only if $\tau$ is pseudoholomorphic: its differential is complex linear
with respect to $J$ and the complex structure on $Z_n$.
We apply this criterion to the following construction.}

\begin{example}[{\bf Affine BSV tori}]
{Let $(M,g)=V/\Lambda$ be a flat oriented real $2n$-torus, and choose
an integer $0<r<n$ and an orthogonal decomposition $V=H\oplus K$ with
$\dim_{\mathbb R}H=2r$. Write $\Gamma=P_H(\Lambda)$, where
$P_H:V\to H$ is the orthogonal projection.} Suppose that
\begin{enumerate}
\item $K$ is $\Lambda$-rational, namely, $\Lambda_K=\Lambda\cap K$ is a full lattice in $K$;
\item $J_0$ is an orthogonal complex structure on $H$ that makes $B=H/\Gamma$ an abelian variety; and
\item $f: B \rightarrow Z_{n-r}$ is a {holomorphic map of generic rank $r$}.
\end{enumerate}
At any $(h,k)\in H\oplus K =V$, consider the almost complex structure $J(h,k)=J_0+ J_{f(x)}$, where $x=[h]\in B$ and $J_{f(x)}$ is the almost complex structure on $K$ corresponding to $f(x)\in Z_{n-r}$. Then $J$ is an orthogonal complex structure on $M$.
\end{example}

Clearly, $J$ descends to $M$ and becomes an almost complex structure
compatible with both the orientation and the flat metric. Denote by
$\pi:M\rightarrow B$ the projection induced by $P_H$. The identity
$d\pi\circ J=J_0\circ d\pi$ and the holomorphicity of $f$ show that the
twistor map $\tau=\iota\circ f\circ\pi$ is pseudoholomorphic, where
$\iota:Z_{n-r}\hookrightarrow Z_n$ adjoins the constant horizontal complex
structure $J_0$. The criterion recalled above therefore makes $J$ integrable;
consequently, $\pi$ and $\tau$ are holomorphic.

The condition that $f$ have generic rank $r=\dim_{\mathbb C}B>0$ both
ensures that $J$ is nonconstant and identifies the base dimension with
the twistor rank. Merely requiring $f$ to be nonconstant would not ensure
the latter when $r>1$. The same generic-rank condition implies
$r\leq\dim_{\mathbb C}Z_{n-r}=\frac12(n-r)(n-r-1)$, giving the numerical
bound in Theorem~\ref{thm1}.

The map $\pi : M \rightarrow B$ is a holomorphic submersion, and each fiber $F_x=\pi^{-1}(x)$ at $x\in B$ is a flat complex torus. All fibers are isometric to the flat $(2n-2r)$-torus $K/\Lambda_K$ as Riemannian manifolds. {As complex tori, however, they do not form an isotrivial family. Indeed, complex structures on the fixed marked torus $K/\Lambda_K$ giving a single biholomorphism class lie in a countable $\operatorname{GL}(\Lambda_K)$-orbit. If the family were isotrivial, the connected image $f(B)$ would lie in this countable set and hence be a point, contradicting its positive generic rank.}

{The twistor rank is intrinsic to the Hermitian structure. The vertical
subspace $K$ can also be recovered from $J$: set}
$$ S_J =\{ w\in V \mid dJ_v(w)=0 \ \forall \ v\in V\}. $$
It is a linear subspace of $V$ and it contains $K$ since $J$ is constant along the directions of $K$. If $0\neq u\in H$ also lies in $S_J$, then it descends to a tangent direction of $B$ which lies in the kernel of $df$, where $f: B \rightarrow Z_{n-r}$ defines $J$. Since we assumed that $df$ is injective at a generic point of $B$, we see that this is impossible, hence $S_J=K$.

\begin{proposition} \label{prop1}
Suppose $F: (V/\Lambda ,J,g) \rightarrow (V'/\Lambda' , J', g')$ is a holomorphic isometry between two affine BSV tori. Then there exists an isometry $\widetilde{F}:V \rightarrow V'$ of the form $\widetilde{F}(v)=Av+b$, where $A$ is orthogonal, such that $A\Lambda =\Lambda'$, $AH=H'$, $AK=K'$, and $F$ induces a holomorphic isometry $F_B:B\rightarrow B'$ such that $F_B\circ \pi = \pi' \circ F$ and
\[
f'\bigl(F_B(x)\bigr)=A|_K\, f(x)\,(A|_K)^{-1}.
\]
In particular, $f=f'\circ F_B$ after the two twistor targets are identified by $A|_K$.
\end{proposition}

In other words, if two affine BSV tori are holomorphically isometric, then their defining data agree under the orthogonal map $A$ {and its induced identification of the vertical twistor spaces}.

\begin{proof}
Given a holomorphic isometry $F$ between two affine BSV tori, its lift $\widetilde{F}$ on the universal cover is an isometry hence is in the form $\widetilde{F}(v)=Av+b$ for some $A$ orthogonal.  The holomorphicity means that
$$ A \circ J_v = J'_{Av+b} \circ A. $$
Take derivative at a constant direction $w\in V$, we get
$$ A \circ dJ_v(w) = dJ'_{Av+b}(Aw)\circ A. $$
In particular, $AS_J\subset S_{J'}$. By considering the inverse map of $F$, we get the other direction of the inclusion, hence $AS_J= S_{J'}$. From the discussion right before the proposition, we know that $S_J=K$ always holds, so we get $AK=K'$. Since $A$ is orthogonal, $AH=H'$ as well. For any $v\in V$ and any $\lambda \in \Lambda$, $\widetilde{F}(v+\lambda)=\widetilde{F}(v)+\lambda'$ for some $\lambda' \in \Lambda'$. From this one deduces $A\Lambda =\Lambda'$, hence $A\Gamma =\Gamma'$.  Write $b=(b_h,b_k)$, then it is clear that the map $F_B: B \rightarrow B'$ defined by $F_B([h]) = [Ah+b_h]$ satisfies the commutativity properties stated in the proposition. {Restricting $A\circ J_v=J'_{Av+b}\circ A$ to $K$ gives the displayed relation for $f$ and $f'$.} This completes the proof of the proposition.
\end{proof}

{To determine when an affine BSV torus is a product, we examine the
lattice relative to the intrinsic splitting $V=H\oplus K$. The
$\Lambda$-rationality of $K$ implies that $\Gamma$ is a full lattice in
$H$, and gives a short exact sequence of abelian groups}
\begin{equation}
0\longrightarrow \Lambda_K \overset{i}\longrightarrow \Lambda \overset{p} \longrightarrow  \Gamma \longrightarrow  0.
\end{equation}
Although this short exact sequence always {splits as a sequence of abstract abelian groups}, the lattice $\Lambda$ need not split with respect to the orthogonal decomposition $V=H\oplus K$, even after passage to a finite-index sublattice. To see this, let $\sigma :\Gamma \rightarrow \Lambda$ be a section of $p$, namely, a homomorphism such that $p\circ \sigma$ is the identity map of $\Gamma$. For $\gamma \in \Gamma$, write $\sigma (\gamma ) =(\gamma , \sigma_{\gamma})$. If $\tilde{\sigma}$ is another section of $p$, then $\tilde{\sigma}_{\gamma} - \sigma_{\gamma} \in \Lambda_K$ for every $\gamma \in \Gamma$. Hence the homomorphism
\begin{equation}
\rho : \Gamma \rightarrow K/\Lambda_K, \ \ \ \ \gamma \mapsto [\sigma_{\gamma}]
\end{equation}
is independent of the choice of section of $p$ and is well-defined. This $\rho$ is called the {\em vertical translation holonomy} of $\pi : M \rightarrow B$. Note that when $\rho =0$, $\sigma_{\gamma} \in \Lambda_K$ for any $\gamma \in \Gamma$, or equivalently, $(\gamma ,0)\in \Lambda$ for any $\gamma \in \Gamma$. Hence $\Lambda = \Gamma \oplus \Lambda_K$, and $M$ is isometric to the product torus $H/\Gamma \times K/\Lambda_K$ as a Riemannian manifold. {This is a Riemannian product; the complex structure on the vertical factor still varies over the base.} In this case we will call the affine BSV torus $M$ a {\bf product BSV torus}.

\begin{proposition}
{An affine BSV torus $M$ has a finite unbranched cover which, with the
lifted Hermitian structure, is holomorphically isometric to a product BSV
torus if and only if its vertical translation holonomy $\rho$ has finite image.}
\end{proposition}

\begin{proof}
Let $(V/\Lambda, J,g)$ be an affine BSV torus. Suppose $\rho$ has finite image. Then $\tilde{\Gamma} =\ker (\rho)$ has finite index in $\Gamma$, and $\tilde{\Gamma} \times \{ 0\} \subset \Lambda$ by the definition of $\rho$. Write
$$ \tilde{\Lambda} = (\tilde{\Gamma} \times \{ 0\} ) \oplus (\{ 0\} \times \Lambda_K). $$
Then $\tilde{\Lambda} \subset \Lambda$ has finite index and $(V/\tilde{\Lambda}, J, g)$ is a product BSV torus which is a finite unbranched cover of the given affine BSV torus $(V/\Lambda, J,g)$.

Conversely, suppose that there is a finite-index sublattice
$\tilde{\Lambda}\subset\Lambda$ such that the induced cover
$(V/\tilde{\Lambda},J,g)$ is holomorphically isometric to a product BSV
torus $(V'/\Lambda',J',g')$. Put
$\tilde{\Gamma}=P_H(\tilde{\Lambda})$. We have
$\Gamma'=\Lambda'\cap H'$ and $\Lambda'=\Gamma'\oplus\Lambda'_{K'}$.
By the proof of Proposition \ref{prop1}, there is an orthogonal map
$A:V\rightarrow V'$ such that $AK=K'$, $AH=H'$, and
$A\tilde{\Lambda}=\Lambda'$. Hence $A\tilde{\Gamma}=\Gamma'$, and
$\tilde{\Gamma}\subset H$ is a full lattice. Both
$\tilde{\Gamma}\subset\Gamma$ are full lattices in $H$, so
$[\Gamma:\tilde{\Gamma}]<\infty$. Moreover, since $A$ preserves the two
orthogonal splittings and $\Lambda'$ is a product lattice,
$(\gamma,0)\in\tilde{\Lambda}\subset\Lambda$ for every
$\gamma\in\tilde{\Gamma}$. Thus
$\tilde{\Gamma}\subset\ker\rho$, and $\rho$ has finite image.
\end{proof}

\begin{remark}
The above proposition characterizes when a (finite unbranched cover of an) affine BSV torus can be a product BSV torus. For any $n\geq 3$ one can construct examples of affine BSV torus whose holonomy map $\rho$ has infinite image, thus any of its finite unbranched covers is not a product BSV torus. So in the main result of \cite[Theorem 1]{KYZ}, BSV tori should be understood as affine BSV tori instead of product BSV tori.
\end{remark}

Below is  such an example with twistor rank $r=1$, for any $n\ge 3$.

\begin{example}
Let $\{ e_1, e_2\}$ be an orthonormal basis of $H\cong {\mathbb R}^2$, and $\{ e_3, \ldots , e_{2n}\}$ be an orthonormal basis of $K\cong {\mathbb R}^{2n-2}$. Fix an irrational number $a$ and consider the lattice $\Lambda \subset V=H\oplus K$:
$$ \Lambda = {\mathbb Z}(e_1+ae_3) \oplus {\mathbb Z}e_2 \oplus \Lambda_K, \ \ \ \Lambda_K = \bigoplus_{j=3}^{2n} {\mathbb Z}e_j. $$
We have $\Gamma ={\mathbb Z}e_1 \oplus {\mathbb Z}e_2$. Let $J_0$ be the constant almost complex structure on $H$ defined by $J_0e_1=e_2$, and let $f$ be a non-constant holomorphic map from the elliptic curve $B=H/\Gamma$ into $Z_{n-1}$. Define
$$ J_{(h,k)} =J_0 + J_{f([h])}, \ \ \ \ (h,k)\in V. $$
Here as before $[h]$ is the image of $h$ in $B$ under the projection map, and $J_{f(x)}$ denotes the orthogonal almost complex structure on $K$ corresponding to $f(x)\in Z_{n-1}$. This gives us an affine BSV torus $(V/\Lambda, J,g)$ of twistor rank 1.  Take the section $\sigma: \Gamma \rightarrow \Lambda$ defined by
$$ \sigma (e_1)=e_1+ae_3,\ \ \ \sigma (e_2) = e_2. $$
By definition, the holonomy map $\rho$ is given by $\rho (e_1)=[ae_3]$ and $\rho (e_2)=0$ in the real $(2n-2)$-torus $K/\Lambda_K = ({\mathbb R}/{\mathbb Z})^{2n-2}$. Clearly, since $a$ is irrational, the first element has infinite order in the torus, so $\rho$ has infinite image. {Thus no finite unbranched cover, with the lifted Hermitian structure, is holomorphically isometric to a product BSV torus.}
\end{example}

\vspace{0.3cm}

The possible positive twistor ranks are governed by the vertical twistor
map, independently of the lattice-splitting question. For a fixed fiber
dimension $m\geq2$, pulling back the universal
sequence on $Z_m$ gives a maximal-isotropic subbundle of a trivial
quadratic bundle; the full-rank condition is equivalent to ampleness of
the determinant of its dual.

\begin{theorem}[A vector-bundle characterization of positive twistor ranks]
\label{thm:rank-realization}
Fix $m\geq2$, and let $K$ be an oriented Euclidean real vector space of
dimension $2m$. Write $q$ for the complex-bilinear extension of its inner
product to $K_{\mathbb C}$ and $q^\flat:K_{\mathbb C}\to
K_{\mathbb C}^{\vee}$ for the induced isomorphism. For every integer
$r>0$, the following are equivalent:
\begin{enumerate}
\item There exists an affine BSV torus of complex dimension $m+r$ and
twistor rank $r$.
\item There are an abelian variety $B$ of dimension $r$, a rank-$m$
holomorphic vector bundle $F$ on $B$, and a holomorphic subbundle inclusion
$\jmath:F^\vee\hookrightarrow K_{\mathbb C}\otimes\mathcal O_B$ whose image
is fiberwise $q$-isotropic, such that
\begin{equation}\label{eq:rank-bundle}
0\longrightarrow F^\vee\xrightarrow{\,\jmath\,}
K_{\mathbb C}\otimes\mathcal O_B
\xrightarrow{\,\jmath^\vee\circ(q^\flat\otimes1)\,}
F\longrightarrow0
\end{equation}
is exact, the classifying map of $F^\vee$ takes values in the chosen
component $Z_m$, and $\det F$ is ample.
\end{enumerate}
The classifying map associated with \eqref{eq:rank-bundle} is finite onto
its image, and the torus in {\rm (i)} can be chosen with a split lattice.
\end{theorem}

\begin{proof}
Suppose {\rm (i)} holds. After identifying the oriented vertical Euclidean
space with $K$, the defining data give an abelian variety $B$ of dimension $r$
and a holomorphic map $f:B\to Z_m$ of generic rank $r$.
Let $S_m$ and $Q_m$ be the universal subbundle and quotient bundle on
$Z_m$. The quadratic form identifies $Q_m\cong S_m^\vee$, and
$\det Q_m$ is the ample Pl\"ucker line bundle. Pullback of the universal
sequence gives \eqref{eq:rank-bundle} with $F=f^*Q_m$ and
\[
 \det F\cong f^*(\det Q_m).
\]
Let $Y=f(B)$. Since $B$ and $Z_m$ are projective, $f:B\to Y$ is a
generically finite algebraic map of some degree $d>0$. The line bundle
$\det F$ is nef, and the projection formula gives
\[
 \int_B c_1(\det F)^r
 =d\int_Y c_1\bigl((\det Q_m)|_Y\bigr)^r>0.
\]
Thus $\det F$ is nef and big. On an abelian variety every nef and big
line bundle is ample: its translation-invariant semipositive
representative has positive top power and is therefore positive
definite. This proves {\rm (ii)}.

Now suppose {\rm (ii)} holds. The maximal-isotropic subbundle in
\eqref{eq:rank-bundle} defines a holomorphic classifying map
$f:B\to Z_m$, and $\det F\cong f^*(\det Q_m)$. If a fiber $T$ of $f$
were positive-dimensional, then $(\det F)|_T\cong\mathcal O_T$,
contradicting the ampleness of $\det F$. Thus $f$ has finite fibers.
Since $B$ is projective, the induced map $B\to f(B)$ is proper and hence
finite. In particular, $f$ has generic rank $r$.

Write $B=H/\Gamma$, and denote its constant complex structure by $J_B$.
Choose a translation-invariant Hermitian inner product on $H$ and a full
lattice $\Lambda_K\subset K$. On
$(H\oplus K)/(\Gamma\oplus\Lambda_K)$, set
\[
 J_{(h,k)}=J_B\oplus J_{f([h])}.
\]
This structure is lattice-invariant. Its twistor map is the composition
of the projection to $B$, the map $f$, and the holomorphic inclusion
$Z_m\hookrightarrow Z_{m+r}$ obtained by adjoining $J_B$.
It is therefore pseudoholomorphic, so the twistor integrability criterion
makes $J$ integrable. The resulting affine BSV torus has complex
dimension $m+r$, twistor rank $r$, and a split lattice, proving {\rm (i)}.
\end{proof}

Together with Theorem~\ref{thm1}, this equivalence characterizes existence
at rank $r>0$ on some flat torus of real dimension $2(m+r)$, with no
prescribed lattice or metric. If no data in {\rm (ii)} exist, no orthogonal
complex structure of that positive rank exists on any such flat torus;
constant structures have rank zero and are unaffected.

\begin{example}[The rank bound is not sufficient]
The pair $(n,r)=(10,6)$ satisfies the inequality in Theorem~\ref{thm1}
with equality, but no flat real $20$-torus admits an orthogonal complex
structure of twistor rank $6$. By Theorem~\ref{thm1}, such a structure
would give an affine BSV torus of the same dimension and rank after a
finite cover. Its defining data would require an abelian sixfold $B$ and a holomorphic map
$f:B\to Z_4\cong Q^6$ of generic rank $6$, where $Q^6\subset\mathbb P^7$
is the smooth six-dimensional quadric
\cite[Section~2.1, equation~(2.1)]{KuzSpinor}.
The map would be surjective and, by \cite[Lemma~2.1]{DHP}, finite.

Since $Q^6$ is simply connected, \cite[Proposition~5.1]{DHP} would then
force it to be a product of projective spaces. Its Picard group has rank
one, so this would give $Q^6\cong\mathbb P^6$. This is impossible:
their anticanonical bundles are $\mathcal O_{Q^6}(6)$ and
$\mathcal O_{\mathbb P^6}(7)$, respectively, with $\mathcal O(1)$
generating the Picard group in each case.
\end{example}

\vspace{0.3cm}

\section{{Classification and geometric consequences}}\label{sec:classification}

{The proof of Theorem~\ref{thm1} uses the first Chern class of the
tangent bundle and the fibers of the twistor map. We begin with the
relevant bundle identities on $Z_n$.}

{Recall that $Z_n=OGr^+(n,V_{\mathbb C})\cong SO(2n)/U(n)$
parametrizes orthogonal complex structures on the oriented Euclidean
$2n$-space $(V,g)$, with $V_{\mathbb C}=V\otimes\mathbb C$.}
Extend $g=\langle,\rangle$ complex linearly. An element $J\in Z_n$
corresponds to its $(-i)$-eigenspace $W=T_J^{0,1}V$, a maximal isotropic
subspace of $V_{\mathbb C}$ satisfying $W\cap\overline W=0$. With this
convention, the complex structure on $T_JZ_n$ sends a skew-adjoint
endomorphism $\dot J$ anticommuting with $J$ to $J\dot J$. Hence $Z_n$ is a
totally geodesic complex submanifold of $Gr(n,V_{\mathbb C})$, which embeds
in ${\mathbb P}(\Lambda^nV_{\mathbb C})$ via the {Pl\"ucker embedding}.
Denote by $S_n$ the universal subbundle and by $Q_n$ its quotient bundle.
Thus, over $Z_n$,
\begin{equation*}
0 \rightarrow S_n \rightarrow ({\mathcal O}_{Z_n})^{\oplus 2n} \rightarrow Q_n \rightarrow 0.
\end{equation*}
The quadratic form and maximal isotropy identify $Q_n\cong S_n^\vee$.
{Let $\iota_{\rm spin}:Z_n\hookrightarrow\mathbb P(\Delta)$ be the
Cartan, or half-spin, embedding and put
$L_n=\iota_{\rm spin}^*\mathcal O_{\mathbb P(\Delta)}(1)$.  This is the
ample generator of $\operatorname{Pic}(Z_n)$, and the Pl\"ucker polarization
is its square; equivalently \cite{BHH},}
\begin{equation} \label{eq:4}
\det Q_n \cong L_n^{\otimes 2},
\end{equation}

Now suppose that  $(M^n,J,g)$ is a compact Hermitian manifold such that as Riemannian manifold it is isometric to a flat torus $V/\Lambda$. For convenience, we will use the same letter $g$ to denote the flat metric on the torus. For any $p\in M$, consider the map
$$ (V, J_p) \rightarrow V_{\mathbb C}/T^{0,1}_p\!M, \ \ \ \ v \mapsto [v]. $$
The map is complex linear since $v+iJ_pv \in T^{0,1}_p\!M$ for any $v\in V$. Therefore, we have a complex linear smooth bundle isomorphism
\begin{equation} \label{eq:5}
(TM, J) \cong \tau^{\ast} Q_n.
\end{equation}
Note that this bundle isomorphism need not be holomorphic. Combining equations (\ref{eq:4}) and (\ref{eq:5}), we get
\begin{equation} \label{eq:6}
c_1(TM, J) = 2\,\tau^{\ast}h_n, \ \ \ \ \ h_n=c_1(L_n).
\end{equation}

{The following Albanese lemma supplies an effective divisor whose
class can be compared with \eqref{eq:6}.}

\begin{lemma} \label{lemma1}
{Let $M^n$ be a compact complex manifold diffeomorphic to the real
$2n$-torus. If $M$ belongs to Fujiki class $\mathcal C$, then its Albanese
map $a:M\to Alb(M)$ has degree one, and its ramification divisor $D$ satisfies}
\begin{equation} \label{eq:albanese-ramification}
 D\geq 0, \ \ \ K_M\cong {\mathcal O}(D), \ \ \ c_1(TM)=-[D].
 \end{equation}
\end{lemma}

\begin{proof}
Let $X^n$ be a compact K\"ahler manifold  and $\pi: X\rightarrow M$ be a bimeromorphic modification. Then $\pi_1(X)=\pi_1(M)$ and $b_1(X)=b_1(M)=2n$.  Also, by Hartogs theorem, holomorphic $q$-forms on $M$ and $X$ can be identified for any $1\leq q\leq n$:
$$ \pi^{\ast} : H^0(M,\Omega^q) \xrightarrow{\sim} H^0(X,\Omega^q) . $$
In particular, every holomorphic form on $M$ is $d$-closed. The above
isomorphism and the Hodge decomposition for $X$ give
$h^{1,0}(M)=n$; choose a basis
$\varphi_1,\ldots,\varphi_n$ of $H^0(M,\Omega^1_M)$. Let
$\Gamma\cong{\mathbb Z}^{2n}$ be the free part of $H_1(M,{\mathbb Z})$.
The Albanese map $a:M\rightarrow Alb(M)={\mathbb C}^n/\Gamma$ is defined by
$$ p \mapsto [(\int_{p_0}^p \varphi_1, \ldots , \int_{p_0}^p \varphi_n)], $$
where $p_0$ is a fixed point in $M^n$ and the integral is along any path from $p_0$ to $p$. Since holomorphic one-forms and their conjugates span the first cohomology group of $M$, the
real period pairing is an isomorphism, and the integral periods form a full lattice.  Thus
\[
 a_*:H_1(M,\mathbb Z)\xrightarrow{\sim}H_1(Alb(M),\mathbb Z).
\]
Both integral cohomology rings are exterior algebras on degree one, so \(a\)
has topological degree \(\pm1\).  Holomorphic orientation makes the degree
\(+1\).  The pullback of a translation-invariant holomorphic volume form on
\(Alb(M)\) is therefore nonzero; its zero divisor is the effective ramification
divisor \(D\), which proves~\eqref{eq:albanese-ramification}.
\end{proof}

\begin{proposition} \label{prop3}
Let $(M^n,g)$ be a compact Hermitian manifold that is Levi-Civita flat. Then its Kodaira dimension $\mbox{kod}\,(M)$ is either $0$ or $-\infty$. If $\mbox{kod}\,(M)=0$, or if $M^n$ belongs to the Fujiki ${\mathcal C}$ class, then $(M^n,g)$ {has a finite unbranched cover that is a flat complex torus.}
\end{proposition}

\begin{proof}
Kodaira dimension is invariant under finite unbranched covers, and Fujiki
class ${\mathcal C}$ is preserved by such covers (pull back a K\"ahler
modification). Replacing $(M^n,J,g)$ by its Bieberbach torus cover, and
retaining the same notation, we may therefore assume that its underlying
Riemannian manifold is a flat real torus $V/\Lambda$. By (\ref{eq:6}),
$c_1(TM)=2[\beta]$, where $\beta=\tau^*\alpha$ and $\alpha$ is a closed
positive $(1,1)$-form on $Z_n$ representing the first Chern class of the
ample line bundle $L_n$. Denote by $\omega$ the fundamental form of $g$.
By \cite[Theorem 3]{YangZ}, $g$ is balanced, namely, $\omega^{n-1}$ is
closed. Therefore
$$  c_1(TM) \cdot [\omega^{n-1}] = \int_M 2\beta \wedge \omega^{n-1} \geq 0. $$
Suppose first that $\mbox{kod}\,(M)\geq0$. For some $q>0$ there is a nonzero section of $qK_M$ with effective divisor of zeros $D_q$ ($D_q=0$ as a divisor when the section has no zeros). Hence
\[
 c_1(TM)=-\frac1q[D_q].
\]
Pairing this identity and $c_1(TM)=2[\beta]$ with the closed positive form $\omega^{n-1}$ gives
\[
 0\leq 2\int_M\beta\wedge\omega^{n-1}
 =-\frac1q\int_{D_q}\omega^{n-1}\leq0.
\]
Thus $\beta=0$: a nonzero semipositive $(1,1)$-form has positive pairing with $\omega^{n-1}$ on an open set. Since $\alpha$ is positive, $\tau$ is constant, so $(M,J,g)$ is a flat complex torus and has Kodaira dimension $0$. This proves that the only possibilities are $0$ and $-\infty$.

If $M$ is in Fujiki class ${\mathcal C}$, Lemma \ref{lemma1} gives $c_1(TM)=-[D]$ for an effective divisor $D$. Therefore
$$ c_1(TM) \cdot [\omega^{n-1}] = - \int_D \omega^{n-1} \leq 0, $$
{The same positivity argument forces both sides to vanish, so $\beta=0$ and $\tau$ is constant. Thus $M$ is a flat complex torus.} This completes the proof of the proposition.
\end{proof}

\begin{lemma} \label{lemma2}
Let $A$ be a complex torus of complex dimension $n$, and let $X\subset A$ be a compact {connected} complex submanifold. If $c_1(TX)=0$, then $X$ is {a translate of} a complex subtorus.
\end{lemma}

\begin{proof}
Let $g$ be a flat K\"ahler metric on $A$, and let $h$ be its restriction to $X$. {Put $m=\dim_{\mathbb C}X$. The Gauss--Chern formula expresses the Ricci form of $h$ as a nonpositive form determined by the second fundamental form.}

Fix any $p\in X$. {Choose a local unitary frame $\{e_1,\ldots,e_n\}$ on a neighborhood $U\subset A$ of $p$, with $e_1,\ldots,e_m$ tangent to $X$ along $U\cap X$.} Let $\varphi$ be the local unitary coframe of $(1,0)$ forms dual to $e$. Denote by $\theta$ the matrix under $e$ of the Levi-Civita connection $\nabla$, namely, $\nabla e_i = \sum_{j=1}^n \theta_{ij}e_j$ for each $i$. Write
$$ \theta = \left[ \begin{array}{cc} \theta' & \psi \\ - \psi^{\ast} & \theta'' \end{array} \right] , \ \ \ \ \theta' =(\theta_{ij})_{1\leq i,j\leq m}.$$
{The pullbacks of $\varphi_{m+1},\ldots,\varphi_n$ to $U\cap X$ vanish.
Since $X$ is a complex submanifold, each entry of $\psi$, restricted to
$X$, is a $(1,0)$-form.} Since $g$ is flat, we have $d\theta -\theta \wedge \theta =0$. Its upper left corner block gives
$ \Theta' = d\theta' - \theta' \wedge \theta' = - \psi \wedge \psi^{\ast}$, hence
$$ \mbox{tr}(\Theta') = - \mbox{tr}(\psi \wedge \psi^{\ast}) = - \sum_{i=1}^m \sum_{\alpha=m+1}^n \theta_{i\alpha} \wedge \overline{\theta_{i\alpha}} , $$
where each $\theta_{i\alpha}$ is a local $(1,0)$-form on $X$. The form $\frac{i}{2\pi}\,\mbox{tr}(\Theta')$ represents $c_1(TX)$. {After wedging with $h^{m-1}$ and integrating over $X$, the assumption $c_1(TX)=0$ gives the integral of a nonpositive multiple of $\sum_{i,\alpha}|\theta_{i\alpha}|^2$ equal to zero.} Hence every $\theta_{i\alpha}$ vanishes, so $X\subset A$ is totally geodesic and therefore a translate of a complex subtorus.
\end{proof}

\Needspace{6\baselineskip}

\begin{proof}[{\bf Proof of Theorem \ref{thm1}.}]

By the Bieberbach theorem, $(M,J,g)$ has a finite unbranched cover whose
underlying Riemannian manifold is a flat $2n$-torus. Pull back $J$ and $g$
to this cover. Since the theorem classifies precisely such a finite cover,
we may write it as $M=V/\Lambda$. Denote by $\tau:M\rightarrow Z_n$ its
twistor map and by $r=\dim_{\mathbb C}\tau(M)$ its twistor rank. If $r=0$,
then $M$ is a flat complex torus, so assume $r>0$.

The case $r=n$ cannot occur. Indeed, $\tau$ would then be generically finite
onto the projective variety $\tau(M)$, so $M$ would be Moishezon and hence in
Fujiki class $\mathcal C$. Since $M$ is already a real torus, the proof of
Proposition \ref{prop3} would force $\tau$ to be constant, contradicting
$r=n>0$. Thus $0<r<n$.

The image $\tau(M)$ is an $r$-dimensional projective
variety. Take the Stein factorization $\tau=\nu\circ\sigma$:
$$ M \xrightarrow{\sigma }S \xrightarrow{\nu} \tau(M), $$
where $S$ is an irreducible normal projective variety of dimension $r$,
$\nu$ is a finite morphism, and each fiber of $\sigma$ is connected.

Let $S'\subset S$ be the complement of the singular locus of $S$ and the
critical value set of $\sigma$. Then $S'$ is open and dense in $S$, and
$\sigma' : M' \rightarrow S'$ is a proper holomorphic submersion with
connected fibers. Here $M'=\sigma^{-1}(S')$ and $\sigma'=\sigma|_{M'}$.
{Moreover, $S'$ is connected: the regular locus of an irreducible normal
variety is connected, and deleting a proper complex-analytic subset does
not disconnect it.}

Fix any $s\in S'$, and let $X=X_s=\sigma^{-1}(s)$ be a fiber of $\sigma'$.
Then $\tau(X)=\nu\circ\sigma(X)$ is a point $J_0\in Z_n$.
Let $A=(V/\Lambda,J_0)$ be the complex torus. Then $X\subset A$ is a compact
complex submanifold. {Since $X$ is a fiber of the submersion $\sigma'$,
its normal bundle in $M'$ is trivial.} Hence by (\ref{eq:6}),
$$ c_1(TX)= \iota^{\ast} c_1(TM) = 2\,\iota^{\ast} \tau^{\ast} h_n = 2 \,(\tau\circ \iota )^{\ast} h_n=0, $$
where $\iota:X\rightarrow M$ is the inclusion map.

By Lemma~\ref{lemma2}, each $X_s$ is a translate of a complex subtorus.
Thus there are a point $x_s\in M$ and a real $(2n-2r)$-dimensional linear
subspace $K_s\subset V$ such that
\[
 X_s=x_s+K_s/(K_s\cap\Lambda).
\]
The subspace $K_s$ is $\Lambda$-rational, meaning that $K_s\cap\Lambda$
is a full lattice in $K_s$.

Local sections make the map $s\mapsto K_s$ continuous in the real
Grassmannian $Gr_{\mathbb R}(2n-2r,V)$. Its image is connected and lies
in the countable set of $\Lambda$-rational subspaces, hence is a point.
Thus $K_s$ is independent of $s$ and equals a fixed linear subspace
$K\subset V$.

Denote by $H\subset V$ the orthogonal complement of $K$, and write
$\mathcal K$ and $\mathcal H$ for the corresponding translation-invariant
distributions on $M$. {On} the open dense subset $M'$, we have
$J\mathcal K=\mathcal K$ and $d\tau(\mathcal K)=0$.
By continuity, these identities extend to all of $M$.
{Orthogonality of $J$ also gives $J\mathcal H=\mathcal H$.}
The constant vector fields tangent to $\mathcal K$ are complete,
so $d\tau(\mathcal K)=0$ integrates to invariance under every $K$-translation.

Let $P_H:V\rightarrow H$ be the orthogonal projection, and write
$\Gamma=P_H(\Lambda)$. Since $\Lambda_K:=K\cap\Lambda$ is a full lattice
in $K$, $\Gamma$ is a full lattice in $H$.
{By the $K$-translation invariance, $J$ descends to an almost complex
structure $\tilde J$ on $B=H/\Gamma$, and $P_H$ induces a smooth submersion
$\pi:M\to B$. The Nijenhuis tensor of $\tilde J$ is the projection of that
of $J$, so $\tilde J$ is integrable and $\pi$ is holomorphic.}
The twistor map factors holomorphically through $\pi$, namely,
$\tau=\tilde\tau\circ\pi$, where $\tilde\tau:B\rightarrow Z_n$.

{The map $\tilde\tau$ has generic rank $r$ and is therefore generically
finite onto the projective variety $\tau(M)$.} Hence $B$ is Moishezon and
belongs to the Fujiki class $\mathcal C$.
The induced metric on $B$ is Levi-Civita flat. Since $B$ is already a real
torus, the proof of Proposition \ref{prop3} applies without taking another
cover and forces the twistor map of {the induced Hermitian structure}
$(B,\tilde J)$ to be constant. Thus $\tilde J$ is parallel in the given
flat frame, so $B$ is a complex torus; since it is Moishezon, it is an
abelian variety.

{Finally, put $m=n-r$. Since $J$ preserves the orthogonal decomposition
$V=H\oplus K$, the map $\tilde\tau$ factors through
$Z_r\times Z_m\hookrightarrow Z_n$.}
Its $Z_r$ component is the constant horizontal structure
$\tilde J$, while the vertical factor is a holomorphic map
$f:B\to Z_m$ of generic rank $r>0$. Hence $Z_m$ has positive dimension,
so $m\geq2$, {and}
\[
 r\leq\dim Z_m=\frac12m(m-1).
\]
{The decomposition $V=H\oplus K$, the rationality of $K$, the abelian
variety $B$, and the map $f$ give precisely the defining data of an affine
BSV torus. Thus $M=(V/\Lambda,J,g)$ is an affine BSV torus.}
This completes the proof of Theorem \ref{thm1}.
\end{proof}

\begin{proof}[Proof of Proposition~\ref{prop:bsv-geometry}.]
The metric $g$ is Levi-Civita flat by construction, so it is balanced by
\cite[Theorem~3]{YangZ}. The twistor map is nonconstant; hence the proof of
Proposition~\ref{prop3} gives Kodaira dimension $-\infty$ and excludes
Fujiki class $\mathcal C$. In particular, $M$ is neither K\"ahlerian nor
Moishezon. The holomorphic projection $\pi:M\to B$ onto the
$r$-dimensional abelian variety gives $a(M)\geq r$, while the exclusion of
Moishezon manifolds gives $a(M)<n$.
\end{proof}

{To prove Corollary~\ref{cor2}, we use the relation between the
Levi-Civita connection and the Chern torsion.} Let $(M^n,g)$ be a Hermitian manifold, and $e$ be a local unitary frame with dual coframe $\varphi$. Denote by $T^j_{ik}$ the Chern torsion components under $e$, namely, $T^c(e_i,e_k)=\sum_{j=1}^n T^j_{ik}e_j$. Denote by $\nabla $ and $\nabla^c$ the Levi-Civita and Chern connection of $g$, respectively, and by $\theta$ the Chern connection matrix, namely, $\nabla^ce_i = \sum_j \theta_{ij} \otimes e_j$. From \cite{YangZ}, we get the Levi-Civita connection matrix
 \begin{equation} \label{eq:8}
 \nabla e_i = \sum_j \big\{ \theta^1_{ij}e_j + \overline{\theta^2_{ij}}\overline{e}_j\big\}, \ \ \ \ \theta^1_{ij} = \theta_{ij} + \frac12 \sum_k \big\{ T^j_{ik}\varphi_k - \overline{T^i_{jk}} \overline{\varphi}_k \big\} , \ \ \  \theta^2_{ij} =\frac12 \sum_k \overline{ T^k_{ij} }\varphi_k.
 \end{equation}

\begin{proof}[{\bf Proof of Corollary \ref{cor2}.}]

Let $M^n=(V/\Lambda , J,g)$ be an affine BSV torus of twistor rank $r$. Let $V=H\oplus K$ be the decomposition into horizontal and vertical directions. By definition,  $J$ is invariant under $K$-translations, and when restricted on the horizontal distribution ${\mathcal H}$, $J$ is constant. Therefore we can locally choose a unitary frame $e$ such that $\{ e_1, \ldots , e_r\}$ spans ${\mathcal H}^{1,0}$, $\nabla e_i=0$ for each $1\leq i\leq r$, and $\nabla_X e_{\alpha} =0$ for each $r+1\leq \alpha \leq n$ and any $X\in {\mathcal K}$. For any $1\leq i\leq r$ and any $1\leq a\leq n$, by (\ref{eq:8}) we have
$$ 0 = \langle \nabla e_i , e_a\rangle = \overline{\theta^2_{ia}} = \frac12 \sum_b T^b_{ia} \overline{\varphi}_b. $$
Hence $T^b_{ia}=0$ for any $1\leq i\leq r$ and any $1\leq a,b\leq n$. Similarly, for any $r+1\leq \alpha, \beta,\gamma \leq n$, by $\nabla_{\overline{e}_{\beta}} e_{\alpha} =0$ and (\ref{eq:8}) we get
$$ 0 =\langle \nabla_{\overline{e}_{\beta}} e_{\alpha} , e_{\gamma} \rangle = \overline{ \theta^2_{\alpha \gamma }(e_{\beta})} = \frac12 T^{\beta}_{\alpha \gamma}.$$
So $ T^{\beta}_{\alpha \gamma}=0$ for any $r+1\leq \alpha, \beta,\gamma \leq n$. Combine this with the previous $T^{\ast}_{i\ast}=0$ for any $1\leq i\leq r$, we know that the only possibly nonzero Chern torsion components are $T^i_{\alpha \beta}$, where $1\leq i\leq r$ and $r<\alpha , \beta \leq n$. In particular, $T^{\alpha}_{\ast \ast}=0$ means that the image of $T$ is contained in ${\mathcal H}^{1,0}$, while $T^{\ast}_{i\ast}=0$ means that ${\mathcal H}^{1,0}$ is contained in the kernel $\mbox{ann}(T)$. This completes the proof of Corollary \ref{cor2}.
\end{proof}

\section{{Pluriclosed metrics}}\label{sec:pluriclosed}

\begin{proof}[Proof of Theorem~\ref{thm:no-pluriclosed}.]
Intuitively, variation of the fibers' complex structures produces a
nonzero semipositive $dd^c$-exact $(n-1,n-1)$-form, which is incompatible
with a pluriclosed metric on a compact manifold.

Put $m=n-r$.  The rank condition implies $m\geq2$.  On the universal
cover, use the fixed orthogonal splitting $V=H\oplus K$ and write
\[
 J=J_0\oplus I(h),\qquad
 \omega=\omega_H+\omega_K,
\]
where $I:B\to Z_m$ is the vertical twistor map and $\omega$ is the
fundamental form of the flat metric.  These tensors, as well as the fixed
vertical Hodge operator $\ast_K$, are invariant under the affine deck
translations and therefore descend to $M$; no splitting of the lattice is
needed.

We use the convention $d^c=i(\bar\partial-\partial)$, so that
$dd^c=2i\partial\bar\partial$.  Let $X$ be a constant horizontal unit vector,
put $Y=J_0X$, and set $A=D_XI$.  Holomorphicity of $I$ and orthogonality give
\[
 D_YI=IA,\qquad IA+AI=0,\qquad A^*=-A.
\]
In flat coordinates $X=\partial_x$, $Y=\partial_y$, differentiation of
$I_y=II_x$ yields
\[
 A_y=A^2+I I_{xx},\qquad
 I_yA+IA_y=2IA^2-I_{xx}.
\]
For constant vertical vectors $u,v$, direct application of
$d^c\alpha(U,V,W)=-d\alpha(JU,JV,JW)$ to the real $(1,1)$-form
$\alpha=\omega_K$ now gives
\begin{equation}\label{eq:vertical-hessian}
 \begin{split}
 (dd^c\omega_K)(X,Y,u,v)
 &=-g\bigl((I_{xx}+I_yA+IA_y)u,v\bigr)\\
 &=-2g(IA^2u,v)=2g(IP_Xu,v),
 \end{split}
\end{equation}
where $P_X=A^*A=-A^2\geq0$.  In particular,
\[
 (dd^c\omega_K)(X,Y,u,Iu)=2|Au|^2.
\]

Extend the vertical Hodge star to mixed forms by
$\ast_K(\alpha_H\wedge\beta_K)=\alpha_H\wedge\ast_K\beta_K$. The operator
$\ast_K$ is constant in the flat frame.  On the
$K$-translation-invariant forms used here, $d$ has only horizontal
derivatives, so $\ast_K$ commutes with $d$.  It also commutes pointwise with $J$, because $I(h)$ is an
orientation-preserving isometry of $K$.  Since $\ast_K$ changes total degree
by an even number on the forms at issue, it commutes with $d^c$ as well.
Using $\omega_K^{m-1}=(m-1)!\ast_K\omega_K$, we obtain
\begin{equation}\label{eq:vertical-star}
 dd^c(\omega_K^{m-1})=(m-1)!\ast_K(dd^c\omega_K).
\end{equation}

Only one term in the binomial expansion of $\omega^{n-2}$ survives after
applying $dd^c$: the term with horizontal degree $2r-2$ and vertical degree
$2m-2$.  Indeed, $\omega_H$ is constant, $\omega_K^m=m!\operatorname{vol}_K$
is constant, and any term already of top horizontal degree is killed by the
two new horizontal differentials.  Hence
\begin{equation}\label{eq:positive-current-expansion}
 dd^c(\omega^{n-2})
 =\binom{n-2}{r-1}\omega_H^{r-1}\wedge
 dd^c(\omega_K^{m-1}).
\end{equation}

Choose a horizontal orthonormal frame
$X_1,J_0X_1,\ldots,X_r,J_0X_r$, put $A_a=D_{X_a}I$, and set
\[
 P=\sum_{a=1}^r A_a^*A_a,
 \qquad \omega_P(u,v)=g(IPu,v).
\]
Taking the horizontal trace in \eqref{eq:positive-current-expansion} and
using \eqref{eq:vertical-hessian}--\eqref{eq:vertical-star} gives
\begin{equation}\label{eq:positive-current}
 \begin{aligned}
 dd^c(\omega^{n-2})
 &=c_{r,m}\,\omega_H^r\wedge\ast_K\omega_P,\\
 c_{r,m}
 &=\frac{2(m-1)!}{r}\binom{n-2}{r-1}
 =\frac{2(n-2)!}{r!}>0.
 \end{aligned}
\end{equation}
The form $\omega_P$ is semipositive of type $(1,1)$ on $K$, so
$\ast_K\omega_P$ is a semipositive $(m-1,m-1)$-form.  Thus the right-hand
side of \eqref{eq:positive-current} is a semipositive
$(n-1,n-1)$-form.  It is nonzero on the dense open set where $I$ has rank
$r$, because there $P\neq0$.

If $\eta$ were a pluriclosed Hermitian form on $M$, positivity and compactness
would give
\[
 0<\int_M\eta\wedge dd^c(\omega^{n-2}).
\]
On the other hand, integration by parts and $dd^c\eta=0$ give
\[
 \int_M\eta\wedge dd^c(\omega^{n-2})
 =\int_M dd^c\eta\wedge\omega^{n-2}=0,
\]
a contradiction.
\end{proof}

Note that for $n=3$ and $r=1$, this obstruction was obtained by Khan, Yang, and
Zheng \cite[Lemma~6]{KYZ}. The argument above extends it to arbitrary
vertical dimension and twistor rank, including nonsplit affine lattices,
by combining the vertical Hodge-star identity \eqref{eq:vertical-star}
with the standard positive-current obstruction \cite{AB,Verbitsky}.

\vspace{0.5cm}

\noindent\textbf{Acknowledgments.}
The third named author would like to thank Gabriel Khan, Bo Yang, Xiaokui Yang, and Quanting Zhao for their interest, encouragement, and helpful discussions.

\vspace{0.3cm}

\noindent\textbf{Generative AI disclosure.}
During the development of this work, the authors used ChatGPT and Codex to assist with exploratory computations, possible proof directions, and editorial polishing. The authors wrote the mathematical proofs and computations, and take full responsibility for the contents of the paper.

\vspace{0.3cm}

\noindent\textbf{Declaration on competing interests.}
All authors declare that there are no competing interests for this paper.


\begin{thebibliography}{99}


\bibitem{AB} {L.~Alessandrini and G.~Bassanelli,
\emph{Positive $\partial\bar\partial$-closed currents and non-K\"ahler geometry,}
J. Geom. Anal. \textbf{2} (1992), no.~4, 291--316.}


\bibitem{BHH} F.~Balogh, J.~Harnad, and J.~Hurtubise, \emph{Isotropic Grassmannians, Pl\"ucker and Cartan maps,} J. Math. Phys. \textbf{62} (2021), 021701, doi:10.1063/5.0021269





\bibitem{Bismut} J.-M.~Bismut, \emph{A local index theorem for non-K\"ahler manifolds,}
Math. Ann. \textbf{284} (1989), no.~4, 681--699.



\bibitem {Boothby} W.~Boothby, \emph{Hermitian manifolds with zero curvature,} Michigan Math. J. {\bf 5} (1958), no.~2, 229--233.


\bibitem {BSV} L.~Borisov, S.~Salamon, and J.~Viaclovsky, \emph{Twistor geometry and warped product orthogonal complex structures,} Duke Math. J. \textbf{156} (2011), 125--166.

\bibitem{DHP} {J.-P.~Demailly, J.-M.~Hwang, and T.~Peternell,
\emph{Compact manifolds covered by a torus,}
J. Geom. Anal. \textbf{18} (2008), no.~2, 324--340.}


\bibitem {FinoVezzoni}  A.~Fino, and L.~Vezzoni, \emph{Special Hermitian metrics on compact solvmanifolds,} J. Geom. Phys. \textbf{91} (2015), 40--53.

\bibitem {FinoVezzoni1}  A.~Fino, and L.~Vezzoni, \emph{On the existence of balanced and SKT metrics on nilmanifolds,} Proc. Amer. Math. Soc., {\bf 144} (2016), no.~6, 2455--2459.





\bibitem {FuZhou} J.-X.~Fu and X.~Zhou, \emph{Scalar curvatures in almost Hermitian geometry and some applications,} Sci. China Math. \textbf{65} (2022), 2583--2600.

\bibitem{Fujiki} A.~Fujiki, \emph{On Automorphism Groups of Compact K\"ahler Manifolds,} Invent. Math. \textbf{44} (1978), no.~3, 225--258.



\bibitem{KYZ} G.~Khan, B.~Yang, and F.~Zheng, \emph{The set of all orthogonal complex structures on the flat 6-tori,} Adv. Math. \textbf{319} (2017), 451--471.

\bibitem{KuzSpinor} {A.~Kuznetsov,
\emph{On linear sections of the spinor tenfold. I,}
Izv. Math. \textbf{82} (2018), no.~4, 694--751.}


\bibitem {LS} R.~Lafuente and J.~Stanfield, \emph{Hermitian manifolds with flat Gauduchon connections,} Annali della Scuola Normale Superiore di Pisa, Classe di Scienze, 2024, Vol XXV, no.~4, 2241--2258.



\bibitem  {Milnor} J.~Milnor, \emph{Curvatures of left invariant metrics on Lie groups,} Advances in Math. \textbf{21} (1976), no.~3, 293--329.



\bibitem {Pittie} H.~Pittie, \emph{The Dolbeault-cohomology ring of a compact, even-dimensional Lie group,} Proc. Indian Acad. Sci. Math. Sci. \textbf{98} (1988), 117--152.


\bibitem {Salamon} S.~Salamon, \emph{Orthogonal complex structures, Differential geometry and applications,} (Brno, 1995), 103--117, Masaryk Univ., Brno, 1996.

\bibitem {Samelson} H.~Samelson, \emph{A class of complex analytic manifolds,} Portugaliae Math. \textbf{12} (1953), 129--132.




\bibitem{Strominger} A.~Strominger, \emph{Superstrings with torsion},
Nuclear Phys. B \textbf{274} (1986), no.~2, 253--284.


\bibitem{Verbitsky} {M.~Verbitsky,
\emph{Rational curves and special metrics on twistor spaces,}
Geom. Topol. \textbf{18} (2014), no.~2, 897--909.}



\bibitem {WYZ} Q.~Wang, B.~Yang, and F.~Zheng, \emph{On Bismut flat manifolds,}  Trans. Amer. Math. Soc. {\bf 373} (2020), 5747--5772.




\bibitem {YangZ} B.~Yang and F.~Zheng, {\emph{On curvature tensors of Hermitian manifolds,}} Comm. Anal. Geom. \textbf{26} (2018), no.~5, 1195--1222.


\bibitem {YangZ1} B.~Yang and F.~Zheng, {\emph{On compact Hermitian manifolds with flat Gauduchon connections,}} Acta Math. Sinica (English Series). \textbf{34} (2018), 1259--1268.



\bibitem{ZhaoZ} Q.~Zhao and F.~Zheng, \emph{On Gauduchon K\"ahler-like manifolds,}  J. Geom. Anal. \textbf{32} (2022), no.~4, Paper No.~110, 27 pp.

\end{thebibliography}
\end{document}